\documentclass[10pt]{article}
\usepackage{amssymb}
\usepackage{amsmath}
\usepackage{amsfonts}
\usepackage{amsthm}

\usepackage[english]{babel}
\usepackage[affil-it]{authblk}

\newtheorem{theorem}{Theorem}
\newtheorem{lemma}{Lemma}
\newtheorem{proposition}{Proposition}
\newtheorem{corollary}{Corollary}

\title{Minimum supports of eigenfunctions of Johnson graphs\thanks{This research was financed by the Russian Science Foundation (grant No 14-11-00555) }}

\author{Konstantin Vorob'ev, Ivan Mogilnykh, Alexandr Valyuzhenich, %
  \thanks{E-mail address: \texttt{vorobev@math.nsc.ru, ivmog@math.nsc.ru, graphkiper@mail.ru}}}
\affil{Sobolev Institute of Mathematics, pr. Akademika Koptyuga 4,
Novosibirsk 630090, Russia}
\begin{document}

\maketitle

\begin{abstract}
We study the weights of eigenvectors of the Johnson graphs
$J(n,w)$. For any $i \in \{1,\ldots,w\}$ and sufficiently large
$n, n\geq n(i,w)$ we show that an eigenvector of $J(n,w)$ with the
eigenvalue $\lambda_i=(n-w-i)(w-i)-i$ has at least
$2^i(^{n-2i}_{w-i})$ nonzeros and obtain a characterization of
eigenvectors that attain the bound.
\end{abstract}

\section{Introduction}

  Let $G=(V,E)$ be an undirected graph.
 A real-valued nonzero function $f:V\rightarrow R$ is called
 a $\lambda$-{\it eigenfunction} of $G$ if the following equality holds for any
 $x\in V$:

 $$\lambda f(x)=\sum_{y\in V:(x,y)\in E} f(y). $$
In other words, $f$ is a $\lambda$-eigenfunction of $G$ if its
vector of values $\overline{f}$ is an eigenvector of the adjacency
matrix $A_G$ of $G$ with eigenvalue $\lambda$ or $\overline{f}$ is
the all-zero vector, i.e. the following holds:

$$A_G\overline{f}=\lambda \overline{f}.$$

The vertices of the {\it Hamming graph} $H(n)$ are the binary
vectors of length $n$, where two vectors are adjacent if they
differ in exactly one coordinate position. Given a pair of vectors
$x$ and $y$ of length $n$ {\it the Hamming distance} $d(x,y)$
between a pair of vectors $x$ and $y$ of length $n$ is the Hamming
graph distance between $x$ and $y$, i.e. the number of positions
at which the corresponding symbols are different. {\it The
support} of a real-valued function (or vector) $f$ is denoted by
$supp(f)$ is the set of nonzeros of $f$. {\it The weight} $wt(x)$
of a vector $x$ is the number of nonzero symbols of $x$.

 The vertices of the Johnson graph $J(n,w)$ are the binary
vectors
 of length $n$ with $w$ ones, where two
vectors are adjacent if they have exactly $w-1$ common ones. Note
that the vertices of $J(n,w)$ are vertices of $H(n)$ of weight
$w$, with the Johnson graph distance being equal halfed Hamming
graph distance.

Various combinatorial objects with extreme characteristics could
be defined in terms of eigenfunctions with certain restrictions.
  In particular, several important notions, such as $(w-1)-(n,w,1)$-designs (including
Steiner triple and quadruple systems), equitable 2-cell partitions
and perfect codes could be defined as eigenfunctions of Johnson
graphs \cite{Cvetkovic}, \cite{AvgMog}. \cite{martin}. The
symmetric difference of a pair of such objects (for example,
Steiner triple systems) is a bitrade \cite{krotov}. In case of the
Johnson graphs, bitrades of small size play an important role in
the classification and characterization problems (for example, see
\cite{assmus}, \cite{krotov}, \cite{zin}) and proved to be a
useful constructive tool for Steiner triple and quadruple systems
\cite{assmus}, \cite{aliev}. Moreover, the topic of the current
paper is related to the question of existence of 1-perfect codes
in different graphs which is one of the most captive problems in
combinatorial coding theory. For $n\leq 2^{250}$ it is known that
no such codes exist in the Johnson graphs $J(n,w)$, see
\cite{gordon}. The study of bitrades of 1-perfect codes may lead
to an improvement of this problem. In this light, the question of
finding the size of minimum support of eigenfunctions of Johnson
graphs for arbitrary fixed eigenvalues is tempting and intriguing.

For surveys on combinatorial objects connected with
eigenfunctions, their bitrades and general theory the reader is
referred to the works of
 Krotov et. al \cite{krotov}, \cite{krotovtezic}, Cho \cite{Cho},
 \cite{Cho2} and the book of Colbourn and Dinitz \cite{CD}.

In the current paper the minimum support question for
eigenfunctions of the Johnson graphs $J(n,w)$ with the eigenvalue 
$(w-i)(n-w-i)-i$ for any $i$, $w$ and $n$, $n\geq n(i,w)$ is
solved and a characterization of minimum support functions is
obtained. The solution for the problem in case of the minimum
eigenvalue is $2^{w}$ \cite{hwang} and the value is attained on a
class of so-called Steiner bitrades \cite{hwang}, \cite{krotov}
that include Pasch-configuration.

\section{Preliminaries}

\subsection{Induced eigenfunctions and eigenvalues of Johnson graphs}

Let $f$ be a real-valued function defined on the vertices of the
Johnson graph $J(n,i)$. Define the function $I^{i,w}(f)$ on the
vertices of $J(n,w)$ as follows:

$$I^{i,w}(f)(x)=\sum_{y, wt(y)=i, d(x,y)=|w-i|} f(y). $$

The function $I^{i,w}(f)$ is called {\it induced} in $J(n,w)$ by
$f$ \cite{AvgMog}. The idea of using induced functions for
representation of the Johnson scheme
 has been exploited since the beginning of its study \cite{Delsarte}. In \cite{AvgMog} the
 concept was
 generalized to a wider class of graphs.

 \begin{theorem}\cite{AvgMog}\label{i_f}

1. Let $f$ be a $\lambda$-eigenfunction of $J(n,i)$. Then if
$i\leq w$ then $I^{i,w}(f)$ is a
$(\lambda+(w-i)(n-i-w))$-eigenfunction of $J(n,w)$.

2. Let $f$ be a real-valued function on the vertices of $J(n,w)$.
Then $I^{w,w-1}(f)\equiv 0$  iff $f$ is a $(-w)$-eigenfunction.
  \end{theorem}
  \begin{proof} The sketch of the proof is done by induction on $w$.
 The second and the first statements of the theorem  for
$i=w-1$  were proven in \cite{AvgMog}, see Theorem 1. In general 
case it is easy to see that for any $f$ we have:
$(w-i)!I^{i,w}(f)=I^{w-1,w}(\ldots (I^{i+1,i+2}(I^{i,i+1}(f))))$,  
which finishes the proof.
\end{proof}

Let $M$ and $M'$ be two
 nonintersecting sets of size $i$ of coordinate positions whose
elements are in a one-to-one correspondence $'$. For a subset $I$
of $M$ by $I'$ denote the set of its images $\{m':m \in
I\}\subseteq M'$. Let the function $f^{i,w,n}$ be defined on the
vectors of weight $w$ and length $n$:

$$f^{i,w,n}(x)=
    (-1)^{|M\cap supp(x)|}, \mbox{if } |supp(x)\cap (M\cup M')|=i \mbox{ and }$$
    $$(supp(x)\cap M)' \cup (supp(x)\cap M')=M',$$
and $f^{i,w,n}(x)=0$ otherwise. The main result of the current
paper is  that $f^{i,w,n}$ is the minimum support eigenfunction of
the Johnson graphs $J(n,w)$ asymptotically on $n$.

\begin{proposition}\label{min_trade_e_f}
The function $f^{i,w,n}$ is a $((w-i)(n-w-i)-i)$-eigenfunction of
$J(n,w)$ with the support of size $2^i(^{n-2i}_{w-i})$.
\end{proposition}
 \begin{proof}

The proof relies on Theorem \ref{i_f}. We show that $f^{i,i,n}$ is
a $(-i)$-eigenfunction of the Johnson graph $J(n,i)$ and that
 $f^{i,w,n}=I^{i,w}(f^{i,i,n})$, where the functions $f^{i,w,n}$ and
 $f^{i,i,n}$ are obtained from the same pair of sets $M$ and $M'$ of sizes $i$.

 Let $x$ be a binary vector of length $n$ and weight
$i-1$. Consider the values of $I^{i,i-1}(f^{i,i,n})$ that could be
expressed as follows:
$$I^{i,i-1}(f^{i,i,n})(x)=\sum_{y: wt(y)=i,supp(x)\subset supp(y)}f^{i,i,n}(y).$$

 We have several cases.  If
 $|(M\cap supp(x))'\cup (M\cap supp(x))|=i-1$ then
there are just two elements $m\in M$ and $m'\in M'$ neither of
which belongs to $supp(x)$. Therefore $I^{i,i-1}(f^{i,i,n})(x)$ is
zero, since exactly two summands, $f^{i,i,n}(y)$ for $y$ such that
$supp(y)=supp(x)\cup \{m\}$ or $supp(y)=supp(x)\cup \{m'\}$ are
equal to $-1$ and $1$ and the other summands are zeros. In the
remaining cases
 every summand $f^{i,i,n}(y)$ is zero, so
$I^{i,i-1}(f^{i,i,n})$ is the all-zero function and by Theorem
\ref{i_f} the function $f^{i,i,n}$ is $(-i)$-eigenfunction of
$J(n,i)$.

We proceed with analogous considerations with the function
$I^{i,w}(f^{i,i,n})$:
$$I^{i,w}(f^{i,i,n})(x)=\sum_{y: wt(y)=i, supp(y)\subseteq
supp(x)}f^{i,i,n}(y).$$

 Again, we have several cases. If $supp(x)\cap (M\cup M')=supp(y)$ for some $y$ such that $f^{i,i,n}(y)\neq
 0$, then the remaining elements of the sum are zeros and we have $I^{i,w}(f^{i,i,n})(x)=f^{i,w,n}(x)$. If there is no $y$ such that $f^{i,i,n}(y)\neq 0$ and $supp(y)\subseteq supp(x)$ then by definition of $f^{i,i,n}$ we have that
$I^{i,w}(f^{i,i,n})(x)=0=f^{i,w,n}(x)$. The remaining case
 where there are several $y$'s such that $supp(y)\varsubsetneq supp(x)\cap (M\cup
 M')$ implies that there are the same number of $-1$'s and $1$'s in
 the sum by definition of $f^{i,i,n}$. Therefore in this case we have that $I^{i,w}(f^{i,i,n})(x)=f^{i,w,n}(x)=0$.

Finally, by Theorem \ref{i_f} we see that
$f^{i,w,n}=I^{i,w}(f^{i,i,n})$ is a
$((w-i)(n-w-i)-i)$-eigenfunction of $J(n,i)$ with the support of
size $2^i(^{n-2i}_{w-i})$.
 \end{proof}

\begin{theorem}\cite{Delsarte}
The eigenvalues of $J(n,w)$ are numbers
$\lambda_i(n,w)=(w-i)(n-w-i)-i, i\in \{0,\ldots,w\}$ with
multiplicities $(^{n}_{i})-(^{n}_{i-1})$.
\end{theorem}

Note that an alternative proof for the previous theorem could be
done with the help of Theorem \ref{i_f} and Proposition
\ref{min_trade_e_f} by induction on $w$.

The minimum support problem for eigenfunctions of a
Johnson graph with the minimum eigenvalues is equivalent to the
problem of minimum size of Steiner bitrades of strength $i-1$ with
blocks of size $i$, see theorem below.

\begin{theorem}\label{T3} \cite{liebzim}, \cite{hwang}, \cite{krotov}
The support of $\lambda_i(n,i)$-eigenfunction of $J(n,i)$ is at
least $2^i$ and  any function that attains the bound is
$f^{i,i,n}$ up to multiplication by a scalar. 
\end{theorem}
The main result of the current paper is that the function
$f^{i,w,n}$ is the minimum $((w-i)(n-w-i)-i)$-eigenfunction of
$J(n,w)$ asymptotically.

\subsection{Reduction lemma }

Here we describe  a way to relate eigenspaces of different Johnson
graphs, which can be useful in providing inductive arguments. A
similar idea 
 was suggested in \cite{val} for studying minimum
support eigenfunctions of q-ary Hamming graphs.

Let $f$ be a real-valued $\lambda_i(n,w)$-eigenfunction of
$J(n,w)$ for some $i\in \{0,1,\dots,w\}$ and $j_1,j_2\in
\{1,2,\dots,n\}$, $j_1<j_2$. Define a real-valued function
$f_{j_1,j_2}$ as follows: for any vertex $y=(y_1,y_2, \dots
,y_{j_1-1},y_{j_1+1}, \dots ,y_{j_2-1},y_{j_2+1},\dots,y_n)$ of
$J(n-2,w-1)$
\begin{flushright}$f_{j_1,j_2}(y)=f(y_1,y_2, \dots
,y_{j_1-1},1,y_{j_1+1}, \dots ,y_{j_2-1},0,y_{j_2+1},\dots,y_n)\,$
\end{flushright} \begin{flushright}
 $-f(y_1,y_2, \dots ,y_{j_1-1},0,y_{j_1+1}, \dots ,y_{j_2-1},1,y_{j_2+1}, \dots ,y_n).$
\end{flushright}

\begin{lemma}\label{L1}
 If f is $\lambda_i(n,w)$-eigenfunction of
$J(n,w)$ then $f_{j_1,j_2}$ is a
$\lambda_{i-1}(n-2,w-1)$-eigenfunction of $J(n-2,w-1)$.
\end{lemma}
\begin{proof}
Without loss of generality we have $j_1=1, j_2=2$. The statement
follows from the fact that vertices $(0,1,y_3,\ldots,y_n)=(0,1,y)$
and $(1,0,y)$ have common neighbors in the subgraphs of $J(n,w)$
induced by sets of vertices $\{(0,0,z): wt(z)=w\}$ and $\{(1,1,z):
wt(z)=w-2\}$. More precisely, we have the following equalities:

$$\lambda_i(n,w)f(0,1,y)=f(1,0,y)+\sum_{z: wt(z)=w-1,
d(z,y)=1}f(0,1,z)+$$
$$\sum_{z: wt(z)=w,
d(z,y)=1}f(0,0,z)+\sum_{z: wt(z)=w-2, d(z,y)=1}f(1,1,z),$$
$$\lambda_i(n,w)f(0,1,y)=f(1,0,y)+\sum_{z: wt(z)=w-1,
d(z,y)=1}f(1,0,z)+$$ $$\sum_{z: wt(z)=w,
d(z,y)=1}f(0,0,z)+\sum_{z: wt(z)=w-2, d(z,y)=1}f(1,1,z),$$

therefore we have that 
$$(\lambda_i(n,w)-1)(f(0,1,y)-f(1,0,y))=(\lambda_i(n,w)-1)f_{1,2}(y)=\sum_{z: d(z,y)=1}f_{1,2}(z)$$

In other words, $f_{1,2}$ is $(\lambda_i(n,w)-1)$-eigenfunction
which taking into account that
$\lambda_i(n,w)-1=\lambda_{i-1}(n-2,w-1)$ finishes the proof.

\end{proof}

As we see, given an eigenfunction $f$ from the reduction Lemma
\ref{L1} we obtain the eigenfunctions $f_{j_1,j_2}$ in a Johnson
graph with smaller parameters for every distinct coordinates 
$j_1,j_2$. In some cases the resulting function $f_{j_1,j_2}$ is
just all-zero function, for example when $f=f^{i,w,n}$ from
Proposition \ref{min_trade_e_f} with $n\ge 2w+2$ and
$j_1,j_2\not\in M\cup M'$.

\begin{lemma}\label{L2}
Let $f$ be a real-valued function of $J(n,w)$, and $j_1,j_2,j_3\in \{1,2,\dots,n\}$, $j_1<j_2<j_3$. If $f_{j_1,j_2}\equiv 0$ and $f_{j_1,j_3}\equiv 0$ then $f_{j_2,j_3}\equiv 0$.
\end{lemma}
\begin{proof}
Without loss of generality we can take $j_1=1$,$j_2=2$ and $j_3=3$. Let us fix $z=(z_1,z_2,\dots,z_{n-2})\in J(n-2,w-1)$.
In these terms our goal is to prove that $$f(z_1,1,0,z_2,\dots,z_{n-2})=f(z_1,0,1,z_2,\dots,z_{n-2}).$$
Since $f_{1,2}\equiv 0$ and $f_{1,3}\equiv 0$, we have
$$f(0,1,1,z_2,\dots,z_{n-2})=f(1,0,1,z_2,\dots,z_{n-2}),$$
$$f(0,1,0,z_2,\dots,z_{n-2})=f(1,0,0,z_2,\dots,z_{n-2}),$$
$$f(0,1,1,z_2,\dots,z_{n-2})=f(1,1,0,z_2,\dots,z_{n-2}),$$
$$f(0,0,1,z_2,\dots,z_{n-2})=f(1,0,0,z_2,\dots,z_{n-2}).$$
Combining the first and the third equalities we obtain
$$f(1,1,0,z_2,\dots,z_{n-2})=f(1,0,1,z_2,\dots,z_{n-2}),$$
and combining the second and the fourth equalitues we find that
$$f(0,1,0,z_2,\dots,z_{n-2})=f(0,0,1,z_2,\dots,z_{n-2}),$$

so $f_{2,3}\equiv 0$.

\end{proof}
\begin{corollary}\label{C1}
Let $f$ be a real-valued function of $J(n,w)$. Then the set
$N=\{1,2,\dots,n\}$ of coordinates can be partitioned into $t(f)$
sets $$N={\bigsqcup_{j=1}^{t(f)}{N_j}}, |N_j|>0,$$ such that the
following properties hold: $$\mbox{ for any } j\in
\{1,2,\dots,t(f)\}\,\mbox{ and }
 j_1,j_2\in N_j \mbox{ we have that } \,f_{j_1,j_2}\equiv 0 \text{,}$$
$$\mbox{ if there are } j_1,j_2 \mbox{ such that }(f_{j_1,j_2}\equiv 0) \mbox{ then  there is } j\in \{1,2,\dots, t(f)\}:\,j_1,j_2\in N_j.$$
\end{corollary}

\section{Main result}

\begin{theorem}\label{main}

Let $i,w$ be positive integers, $w\ge i$. There is $n_0(i,w)$ such
that for all $n\geq n_0(i,w)$
 and any nonzero $\lambda_i(n,w)$-eigenfunction $f$ of
$J(n,w)$ the following holds:
$$|supp(f)|\ge 2^i {n-2i \choose w-i},$$ with equality attained only for
the function $f^{i,w,n}$ from Proposition \ref{min_trade_e_f} up
to multiplication by a scalar.
\end{theorem}
\begin{proof}
The proof is based on the induction on $i$. For $i=0$ the
statement is obviously true. We suppose that the statement of the
theorem is true for all $i'<i$
 and we are to prove it for $i$, arbitrary $w\ge i$ and $n$ big enough.

Suppose that the opposite is true, i.e. for some $i$ and $w$ there
is a non-zero $\lambda_i(n,w)$-eigenfunction with the support of
size less then $2^i {n-2i \choose w-i}$. According to Corollary
\ref{C1} the set $N=\{1,2,\dots,n\}$ of $n$ coordinates can be
partitioned. Without loss of generality we may assume that there
are exactly $n$ parts $S_1,S_2,\dots,S_n$ of sizes $t_1,t_2,
\dots, t_n$ correspondingly with $t_j\ge 0$ that partition $N$:
$N=\cup_{i\in \{1,\ldots,n\}}S_i$, 
 such that $f_{j_1, j_2}\equiv
0$ iff $j_1$ and $j_2$ are in one part. Let us take
$$T=\max_{j=1,2,\dots,n}{t_j},$$ so $T$ is the size of the largest
part, which may be not unique. Again without loss of generality we
can consider $|S_1|=T$. Corollary \ref{C1} yields
\begin{equation}\label{ineq1} |\{(j_1,j_2):j_1>j_2,f_{j_1,j_2}\not\equiv 0\}|=\sum_{1\le k<l\le n}{t_kt_l}\ge T(n-T),\end{equation} because every pair of coordinates from different parts gives us a non-zero function.

Denote by $X$ the set of pairs of adjacent vertices of $J(n,w)$
where $f$ has distinct values:

$$\{(x_1,x_2): wt(x_1)=wt(x_2)=w, |supp(x_1)\cap supp(x_2)|=w-1, f(x_1)\neq f(x_2) \}.$$

A pair of adjacent vertices $x_1$, $x_2$ is uniquely characterized
by the pair of elements $j_1=supp(x_1)\setminus supp(x_2)$ and
$j_2=supp(x_2)\setminus supp(x_1)$. Therefore the size of $X$ is
not less then $$|\{(j_1,j_2):j_1>j_2,f_{j_1,j_2}\not\equiv
0\}||supp(f^{i-1,w,n-2})|.$$

On the other hand the size of $X$ obviously does not exceed
$w(n-w)|supp(f)|$, so

\begin{equation}\label{ineq2} w(n-w)|supp(f)|\ge
|\{(j_1,j_2):j_1>j_2,f_{j_1,j_2}\not\equiv
0\}||supp(f^{i-1,w,n-2})|.
\end{equation}
 Since $|supp(f)|\le 2^i {n-2i \choose w-i}$ and by inductive hypothesis for $n$ big enough the inequality $|supp(f^{i-1,w,n-2})|\ge 2^{i-1} {(n-2)-2(i-1) \choose (w-1)-(i-1)}$ holds, we finally
 obtain
$$2w(n-w)\ge T(n-T).$$
For our following arguments we suppose that \begin{equation}\label{ineq_n}n>2w^2+4w+1,\end{equation} and it gives us $T\le 2w$ or $T\ge n-2w$. \\
Let us start with the first case: $T\le 2w$. Returning to
(\ref{ineq1}) we have $$|\{(j_1,j_2):j_1>j_2,f_{j_1,j_2}\not\equiv
0\}|=\sum_{1\le k<l\le
n}{t_kt_l}=\frac{1}{2}((\sum_{k=1}^{n}{t_k})^2-\sum_{k=1}^{n}{t_k^2}).$$
Since $n=\sum_{j=1}^{n}{t_j}$, we obtain
\begin{equation}\label{ineq3}
|\{(j_1,j_2):j_1>j_2,f_{j_1,j_2}\not\equiv
0\}|=\frac{1}{2}(n^2-\sum_{k=1}^{n}{t_k^2}) .
\end{equation}
Providing the same argument as in proving (\ref{ineq2}) we have
that
$$|\{(j_1,j_2):j_1>j_2,f_{j_1,j_2}\not\equiv
0\}|\le 2w(n-w). $$ 
 As we know, $t_k$'s are real non-negative
numbers, which are not greater than $2w$, therefore we have
$$\sum_{k=1}^{n}{t_k^2}\le \frac{n}{2w}(2w)^2.$$
Combining two previous inequalities and (\ref{ineq3}) we finally get $2wn\ge n^2-4wn+4w^2$, which is not true for $n>(3+\sqrt{5})w$ and consequently for $n>2w^2+4w+1$ too.\\

Now we consider the case  when $T\ge n-2w$. Let us divide the set
of coordinates $N=\{1,2,\dots,n\}$ into two non-intersecting parts
$N_1$ and $N_2$, such that $N_2\subseteq S_1$ and $|N_2|=n-2w$.
Without loss of generality we can take $N_1=\{1,2,\ldots,2w\}$ and
$N_2=\{2w+1,2w+2,\ldots, n\}$. The fact that for any $j_1,j_2\in
N_2$ we have that $f_{j_1,j_2}\equiv 0$ guarantees us that $f(x)$
does not depend on the distribution of ones of vector $x$ in
$N_2$, only on their number in $N_2$. In other words, if $x_1,
x_2\in J(n,w)$ and $supp(x_1)\cap N_2=supp(x_2)\cap N_2$ then
$f(x_1)=f(x_2)$. Based on this property let us define a function
$h:H(2w) \to \mathbb{R}$ as follows: $$h(z)=\begin{cases}
f(z,\underbrace{1,\dots,1}_{w-wt(z)},\underbrace{0,\dots,0}_{n-3w+wt(z)}), wt(z)\le w\\
0,&\text{otherwise.}
\end{cases} $$
The function $f$ is a $\lambda_i(n,w)$-eigenfunction of $J(n,w)$,
therefore for every $x\in J(n,w)$
\begin{equation}\label{eigenJ}
\lambda_i(n,w) f(x)=\sum_{y\in J(n,w):|supp(x)\cap supp(y)|=w-1} f(y).
\end{equation}
Take $x=(z,\underbrace{1,\dots,1}_{w-j},\underbrace{0,\dots,0}_{n-3w+j})$ for arbitrary $z\in H(2w),\,wt(z)=j\le w$ and rewrite (\ref{eigenJ}) in terms of values of $h$:

\begin{equation}\label{eigenH}
 \begin{array}{r@{\,}r@{\,}l@{\,}l}
 \lambda_i(n,w) h(z) =  &\sum\limits_{wt(z')=j, |supp(z)\cap supp(z')|=j-1} &h(z')&+\\
         &\sum\limits_{wt(z')=j-1,|supp(z)\cap supp(z')|=j-1} &(n-3w+j)h(z')&+ \\
         &\sum\limits_{wt(z')=j+1,|supp(z)\cap supp(z')|=j} &(w-j)h(z')&+\\
         &&(w-j)(n-3w+j)h(z).

   \end{array}
\end{equation}

In the final part of the proof we are focused on properties of the
function $h$. The function $f$ is such that $f\not\equiv 0$, so
$h\not\equiv 0$. Let $j$ be a minimal integer, such that there is
$z \in H(2w): wt(z)=j$ and $h(z)\neq 0$. There are four different
cases:
\begin{enumerate}
\item $j=0$. We have that $h(z)=f(\underbrace{0,\dots,0}_{2w},\underbrace{1,\dots,1}_{w},\underbrace{0,\dots,0}_{n-3w})\neq 0$. By our previous arguments we can permute zeros and ones in $N_2$ without changing the value of $f$. Then we have at least ${n-2w \choose w}$ non-zero values of $f$. For $n$ big enough it is greater than $2^i{n-2i \choose w-i}$ and this leads us to a contradiction.
\item $0<j<i$. Let us take any $z_0\in H(2w):wt(z_0)=j-1$ and rewrite (\ref{eigenH}) for $z=z_0$:
$$0=(w-j)\sum_{wt(z')=j,|supp(z_0)\cap supp(z')|=j} h(z').$$
Since $j<i\le w$ and $z_0$ is arbitrary, this equation implies
that $h$ is a $(-j)$-eigenfunction of $J(2w,j)$ by Theorem
\ref{i_f} (second item). Therefore by Theorem \ref{T3} there are at least
$2^j$ vectors $z\in J(2w,j)$, such that $h(z)\neq 0$. As we did in
case $j=0$ we can permute the values of coordinates of $z$ in
$N_2$ without changing the value of $f$. We conclude that there
are at least $2^j{n-2w \choose w-j}$ non-zeros of function $f$ and
for $n$ big enough this value is greater than $2^i{n-2i \choose
w-i}$, which is a contradiction.
\item $j>i$. In this case we have $z\in H(2w)$, $wt(z)=j$. Without loss of generality we can take $$z=(\underbrace{1,\dots,1}_{j},\underbrace{0,\dots,0}_{2w-j})$$ and $$\hat{z}=(\underbrace{1,\dots,1}_{j},\underbrace{0,\dots,0}_{2w-j},\underbrace{1,\dots,1}_{w-j},\underbrace{0,\dots,0}_{n-3w+j}).$$ Consider a function $f_{1,3w-j+1}$. By Lemma \ref{L1} this function is a $\lambda_{i-1}(n-2,w-1)$-eigenfunction of $J(n-2,w-1)$. After deleting two coordinates from $N$ we obtain the set $\{2,3,\dots,3w-j,3w-j+2,3w-j+3,\dots,n\}$. Then we repeat this procedure $i$ times more and by $q$, $q:J(n-2i-2,w-i-1)\to \mathbb{R}$ we denote the function such that: $$q=(\dots((f_{1,3w-j+1})_{2,3w-j+2})\dots)_{i+1,3w-j+i+1}.$$ By Lemma \ref{L1} this function is a $\lambda_{-1}(n-2i-2,w-i-1)$-eigenfunction of $J(n-2i-2,w-i-1)$, in other words just a zero-function. On the other hand,  $q(\underbrace{1,\dots,1}_{j-i-1},\underbrace{0,\dots,0}_{2w-j},\underbrace{1,\dots,1}_{w-j},\underbrace{0,\dots,0}_{n-3w+j-i-1})$ equals a linear combination of values of $h$. All vectors except $z$ in this combination have weight less than $j$, so we conclude that  $q(\underbrace{1,\dots,1}_{j-i-1},\underbrace{0,\dots,0}_{2w-j},\underbrace{1,\dots,1}_{w-j},\underbrace{0,\dots,0}_{n-3w+j-i-1})=h(z)\neq 0$, which contradicts the fact that $q$ is all-zero function.
\item $j=i$. Providing the same arguments as in case $0<j<i$ we prove that there are at least $M\geq 2^i$ and $M{n-2w \choose w-i}\geq 2^i{n-2w \choose w-i}$ non-zero values of $h$ in $J(2w,i)$ and $f$ in $J(n,w)$ correspondingly.
The case $M>2^i$ leads us to a contradiction, because $$M{n-2w \choose w-i}>2^i{n-2i \choose
w-i} $$ for $n$ big enough.

What we are interested now is what one can say about $h$ in case
$M=2^i$. By Theorem \ref{T3} the function $h$ on vertices of
$J(2w,i)$ is $f^{i,i,2w}$ up to a permutation of the first $2w$
coordinates and a multiplication by a scalar, where $M\cup M'=\{1,2,\dots,2i\}$. Without loss of 
generality after dividing by the scalar and applying the
permutation to $h$ we can consider that $h$ is equal to
$f^{i,i,2w}$ on vertices of $J(2w,i)$.  However, we still do not
know the values of $h$ in other vertices of $H(2w)$.

Let us take some $s_0\in \{2i+1,2i+2,\dots,2w\}$ and consider $f_{s_0,2w+1}$. Our following goal is to show that
$f_{s_0,2w+1}\equiv 0$.
Suppose that the opposite is true and take some $x\in J(n-2,w-1)$ with minimal $m=|supp(x)\cap (\{1,2,\dots 2w\}\setminus \{s_0\})|$ such that $f_{s_0,2w+1}(x)\neq 0$. By definition $f_{s_0,2w+1}(x)=f(x')-f(x'')$, where $x',x''\in J(n,w) $ are obtained from $x$ by adding two coordinates ($s_0$ and $2w+1$) with values $1$ and $0$ for $x'$ and $0$ and $1$ for $x''$ correspondingly. Particularly, it means that $|supp(x')\cap \{1,2,\dots 2w\}|=m+1$ and $|supp(x'')\cap \{1,2,\dots 2w\}|=m$. 

In case $m\le i-2$ vectors $x'$ and $x''$ have less than $i$ ones in the first $2w$ coordinates, so we have $f_{s_0,2w+1}(x)=0-0=0$, and we reach a contradiction. 

In case $m=i-1$ the vector $x'$ has exactly $i$ ones in the first $2w$ coordinates and the vector $x''$ has only $i-1$, what gives us $f_{s_0,2w+1}(x)=f(x')-0=f(x')$. However, as we know $f^{i,i,2w}$ has nonzero values only on some vectors with ones on the first $2i$ coordinates. By definition of $s_0$ we have $s_0\in supp(x')$ and $s_0\not \in \{1,2,\dots,2i\}$, so we conclude that $f_{s_0,2w+1}(x)=0$ and reach a contradiction.  

Consequently, one can claim that $m\ge i$.

Let $supp(x)\cap (\{1,2,\dots 2w\}\setminus \{s_0\})=
\{s_1,s_2,\dots,s_m\}$ and define $\hat{x}$ as the vector obtained
from $x$ by deleting 
 coordinates $\{s_1,s_2,\dots,s_i\}$. 

Consider the function
$$g=(\dots((f_{s_0,2w+1})_{s_1,2w+2})\dots)_{s_i,2w+i+1}.$$
Similar to the case $j>i$ this function is a
$\lambda_{-1}(n-2i-2,w-i-1)$-eigenfunction of $J(n-2i-2,w-i-1)$ by
Lemma \ref{L1}, in other words just the all-zero function. 
 On
the other hand, $g(\hat{x})$ equals a linear combination of values
of $f_{s_0,2w+1}$. It is clear, that only one of them is the value
of $f_{s_0,2w+1}$ on the vector with $m$ ones in the first $2w$
positions (vector $x$), and other have less number of ones there.
Therefore we conclude that $g(\hat{x})=f_{s_0,2w+1}(x)\neq 0$ and
find 
 a contradiction.

Since $s_0$ was an arbitrary element of $\{2i+1,2i+2,\dots,2w\}$
we conclude that $\forall j_1,j_2 \in \{2i+1,2i+2,\dots ,n\}$ the
equality $f_{j_1,j_2}\equiv 0$ holds by Lemma \ref{L2}, i.e.
$f(x)$ depends only on the distribution of ones of $x$ in the first $2i$
positions. The knowledge of values of $h$ in $J(2i,i)$ gives us
that there are at least $2^i{n-2i \choose w-i}$ non-zero values of
$f$ in $J(n,w)$. So we conclude that $h$ is a zero-function
outside $J(2i,i)$ and see that $f=f^{i,w,n}$.

\end{enumerate}
\end{proof}

In the proof of the previous theorem we had to take $n$ big enough
several times independently, so finding a good lower bound on
$n_0(i,w)$ is a problem. Even in the case $i=1$ the answer is still unknown.

\section{Acknowledgements}
The authors are grateful to Sergey Goryainov for providing initial examples of minimum eigenfunctions using computer.


\begin{thebibliography}{5}

\bibitem{aliev} I. Sh. o. Aliev, Combinatorial designs and algebras, Sib. Math.
J. 13(3) (1972), pp. 341--348.

\bibitem{assmus}
E. F. Assmus, Jr and H. F. Mattson, On the number of inequivalent
Steiner triple systems, J. Comb. Theory 1(3) (1966), pp. 301--305.


\bibitem{AvgMog} S.V. Avgustinovich, I.Yu. Mogilnykh,
Induced perfect colorings, Siberian Electronic Mathematical
Reports, 8 (2011), pp. 310--316.

\bibitem{Cho}  S. Cho, Minimal null designs of subspace lattice over
finite fields, Linear Algebra and its Applications, 282 (1998),
 pp. 199--220.

\bibitem{Cho2}  S. Cho, On the support size of null designs of finite ranked posets, Combinatorica, 19(4)  (1999), pp. 589--595.

\bibitem{CD} C. J. Colbourn and J. H. Dinitz, Handbook of
Combinatorial Designs, Discrete Mathematics and Its Applications,
Chapman and Hall/CRC, Boca Raton, London, New York, 2006.

\bibitem{Cvetkovic}  D. M. Cvetkovic, M. Doob, H. Sachs, Spectra of graphs, Academic
Press, New York, London, 1980.

\bibitem{Delsarte}  P. Delsarte, An Algebraic Approach to the Association Schemes of Coding Theory, Philips
Res. Rep. Suppl., 10 (1973), pp. 1--97.


\bibitem{gordon}
  M.D. Gordon, Perfect Single Error-Correcting Codes in the Johnson
    Scheme, IEEE Trans. Inform. Theory, 52(10) (2006), pp. 4670--4672.


\bibitem{hwang}  H. L. Hwang, On the structure of (v, k, t) trades, J. Stat.
Plann. Inference, 13 (1986), pp. 179--191.


\bibitem{krotovtezic}
D. S. Krotov, Trades in combinatorial configurations, in Proc.
 12 Int. seminar "Discrete Mathematics and Its
Applications" (Moscow, 20--25 June 2016)  (in Russian) : O. M.
Kasim-Zade (Ed.),
 MSU, Moscow (2016), pp. 84--96,.

\bibitem{krotov}
D. S. Krotov, I. Yu. Mogilnykh, V. N. Potapov, To the theory of
q-ary Steiner and other-type trades, Discrete Math., 339(3)
(2016), pp. 1150--1157.

\bibitem{liebzim}
R.A. Liebler, K.H. Zimmermann, Combinatorial S$_{>n}$-modules as
codes. J. Algebraic Combin., 4 (1995), pp. 47--68.

\bibitem{martin}  W. J. Martin, Completely Regular Designs. Journal of
Combinatorial Designs, 4 (1998), pp. 261--273.



\bibitem{val}
A. A. Valyuzhenich, Minimum supports of eigenfunctions of Hamming
graphs, Discrete Math., 340(5) (2017), pp. 1064--1068.



\bibitem{zin}
V. A. Zinoviev and D. V. Zinoviev. On one transformation of
Steiner quadruple systems S(v, 4, 3), Probl. Inf. Transm., 45(4),
(2009), pp. 317--332.

\end{thebibliography}
\end{document}